\documentclass[11pt,twoside,leqno]{article}
\usepackage{amsmath,amsthm,amssymb}
\numberwithin{equation}{section}
\usepackage[all]{xy}
\usepackage[dvips]{graphicx,color,psfrag}
\usepackage{here}

\title{\bf An explicit prime geodesic theorem for discrete tori and the hypergeometric functions}
\author{Yoshinori YAMASAKI\thanks{Partially supported by Grant-in-Aid for Scientific Research (C) No. 15K04785.}
} 
\date{\today}

\theoremstyle{theorem}

\theoremstyle{definition}
\newtheorem*{multiproclaim}{\variable@name}

\theoremstyle{plain}
\newtheorem{thm}{Theorem}[section]
\newtheorem{prop}[thm]{Proposition}
\newtheorem{lem}[thm]{Lemma}

\newtheorem{conj}[thm]{Conjecture}

\theoremstyle{definition}

\newtheorem{example}[thm]{Example}
\newtheorem{remark}[thm]{Remark}

\newenvironment{MSC}{%
\smallbreak
\noindent \textbf{2010\ Mathematics Subject Classification\,:}}

\newenvironment{keywords}{%
\noindent\textbf{Key words and phrases\,:}\itshape}

\newcommand{\Spec}{\mathrm{Spec}\,}

\renewcommand{\Re}{\mathrm{Re}\,}

\newcommand{\DT}{\mathrm{DT}}
\newcommand{\RT}{\mathrm{RT}}

\begin{document}

\setlength{\baselineskip}{15pt}
\maketitle

\begin{abstract}
 The discrete tori are graph analogues of the real tori, 
 which are defined by the Cayley graphs of a finite product of finite cyclic groups. 
 In this paper, using the theory of the heat kernel on the discrete tori
 established by Chinta, Jorgenson and Karlsson, 
 we derive an explicit prime geodesic theorem for the discrete tori, which is not an asymptotic formula.
 To describe the formula, we need generalizations of the classical Jacobi polynomials,
 which are defined by the Lauricella multivariable hypergeometric function of type $C$.
\begin{MSC}
 {\it Primary}
 11M36 
 {\it Secondary} 
 05C30,
 33C65
\end{MSC} 
\end{abstract}

\section{Introduction and the main results}

 Let $X=(V,E)$ be a graph with $V$ and $E$ being respectively the sets of all vertices and edges of $X$. 
 Throughout of the present paper, we always assume that graphs satisfy the standard conditions,
 that is, they are finite, undirected, simple, connected and regular.
 Assume that $X$ is a $(q+1)$-regular graph. 
 It is well known that, as an analogue of the prime number theorem, 
 we have so-called the prime geodesic theorem for $X$.
 We here briefly review it (more precisely, see e.g., \cite{Terras2011}).
 For a positive integer $n$,
 let $N_{X}(n)$ be the number of all reduced cycles $C$ in $X$ with $l(C)=n$ where $l(C)$ is the length of $C$
 and $\pi_{X}(n)$ the number of all equivalence classes $[P]$ of the prime reduced cycles $P$ in $X$ with $l(P)=n$.
 Here, the cycle $C$ is called reduced if $C^2$ has no backtrack and
 is prime if it can not be expressed as $C=D^f$ for any cycle $D$ and $f\ge 2$.
 It is easy to see that $N_X(n)=\sum_{b\,|\,n}b\pi_{X}(b)$ and hence
 $\pi_X(n)=\frac{1}{n}\sum_{b\,|\,n}\mu(\frac{n}{b})N_{X}(b)$ by the M\"obius inversion formula.
 Let $W$ be the edge adjacency matrix of $X$, which is a square matrix of size $2|E|$, 
 and $\Spec(W)$ the set of all eigenvalues of $W$ with multiplicities.
 Then, it holds that  
\begin{equation}
\label{for:PGTforN}
 N_X(n)
=\sum_{\lambda\in\Spec(W)}\lambda^n
\end{equation}
 and hence 
\begin{equation}
\label{for:PGTforPi}
 \pi_X(n)
\sim \delta_X \frac{q^{n}}{n} \quad (n\to\infty),
\end{equation}
 if $\delta_X\mid n$ (and $\pi_X(n)=0$ otherwise).
 Here, $\delta_X$ is the greatest common divisor of all lengths of prime reduced cycles in $X$.
 Actually, we can obtain \eqref{for:PGTforPi} from \eqref{for:PGTforN} 
 by using the result obtained by Kotani and Sunada \cite{KotaniSunada2000},
 which asserts that the eigenvalues of $W$ with the largest absolute value
 are given by $\lambda=qe^{\frac{2\pi ia}{\delta_X}}$ for $a=1,2,\ldots,\delta_X$.
 As the prime number theorem is obtained via the Riemann zeta function, 
 these formulas are also related to the zeta function $Z_X(u)$, called the Ihara zeta function \cite{Ihara1966},
 associated with $X$ defined by the following Euler product;  
 \[
 Z_{X}(u)
=\prod_{[P]}\bigl(1-u^{l(P)}\bigr)^{-1} \quad (|u|<q^{-1}).
\]
 Here, in the product, $[P]$ runs over all equivalence classes of the prime reduced cycles in $X$.
 In fact, since we have from the definition 
\begin{equation}
\label{for:generating_funcition_of_N}
 u\frac{d}{du}\log Z_{X}(u)=\sum^{\infty}_{n=1}N_X(n)u^n,
\end{equation}
 one obtains \eqref{for:PGTforN} by the following determinant expression of $Z_X(u)$ with respect to $W$;
\[
 Z_X(u)^{-1}=\det\bigl(I_{|E|}-uW\bigr).
\]
 Here, for a positive integer $m$, $I_{m}$ is the identity matrix of size $m$.
 Remark that we will encounter another type of the determinant expression of $Z_X(u)$
 (see \S~\ref{subsec:Ihara_zeta}).

 The aim of this paper is to establish an explicit prime geodesic theorem,
 which are {\it not} an asymptotic formula, for the discrete tori.
 Here, for $M=M^{(d)}=(m_1,\ldots,m_d)\in(\mathbb{Z}_{\ge 3})^d$,
 the discrete torus $\DT^{(d)}_{M}$ of dimension $d$ is defined by
 the Cayley graph of the group $\prod^{d}_{j=1}\mathbb{Z}/m_j\mathbb{Z}$
 associated with the generating set $\{\pm\delta_1,\ldots,\pm \delta_d\}$
 with $\delta_j=(0,\ldots,0,1,0,\ldots,0)\in \prod^{d}_{j=1}\mathbb{Z}/m_j\mathbb{Z}$.
 This is a $2d$-regular graph having $\Vert M\Vert=m_1\cdots m_d$ vertices and $d\Vert M\Vert$ edges.
 Because of the simplicity of the structure of the graph,
 their harmonic analysis are well studied.
 In particular, very recently,
 there are various results on the number of spanning trees, which is sometimes called a {\it complexity},
 of the discrete tori and their degenerated ones
 by establishing the theory of the heat kernel on the graphs
 \cite{ChintaJorgensonKarlsson2010,ChintaJorgensonKarlsson2012,ChintaJorgensonKarlsson2015,Louis2015a,Louis2015b}.

 To state the result, we need a generalization of the Jacobi polynomial:
 For $\alpha=(\alpha_1,\ldots,\alpha_d)\in\mathbb{R}^n$, $\beta\in\mathbb{R}$ and $k\in\mathbb{Z}_{\ge 0}$, define 
\begin{equation}
\label{def:geneJacobi}
 P^{(\alpha,\beta)}_{d,k}(x)
=\frac{(|\alpha|+1)_k}{k!}
 F^{(d)}_C\left(
\begin{array}{c}
 -k,k+|\alpha|+\beta+1\\[3pt] 
 \alpha_1+1,\ldots,\alpha_d+1
\end{array}
;\,\frac{1-x}{2},\ldots,\frac{1-x}{2}
\right),
\end{equation}
 where $(a)_k=\frac{\Gamma(a+k)}{\Gamma(a)}$ is the Pochhammer symbol with $\Gamma(x)$ being the gamma function,
 $|\alpha|=\alpha_1+\cdots+\alpha_d$ and, for $a,b,c_1,\ldots,c_d\in\mathbb{R}$,
\begin{align*}
 F^{(d)}_C\left(\begin{array}{c}a,b\\c_1,\ldots,c_d\end{array};x_1,\ldots,x_d\right)
&=
\displaystyle{\sum_{n=(n_1,\ldots,n_d)\in(\mathbb{Z}_{\ge 0})^d}\frac{(a)_{|n|}(b)_{|n|}}{(c_1)_{n_1}\cdots (c_d)_{n_d}}
\frac{x_1^{n_1}}{n_1!}\cdots \frac{x_d^{n_d}}{n_d!}}\\
& \qquad\qquad\qquad\qquad\qquad\qquad\qquad (|x_1|^{\frac{1}{2}}+\cdots+|x_d|^{\frac{1}{2}}<1)
\end{align*}
 is the Lauricella multivariable hypergeometric function of type $C$.
 Notice that, though $F^{(d)}_C$ is in general an infinite series, 
 $P^{(\alpha,\beta)}_{d,k}(x)$ is actually a polynomial (of degree at most $k$) in $x$
 because of the specialization $a=-k$. 
 Remark that $F^{(1)}_C$ is equal to the Gauss hypergeometric function ${}_2F_{1}$
 and hence $P^{(\alpha,\beta)}_{1,k}(x)$ coincides with the classical Jacobi polynomial.
 When $d=2$, $F^{(2)}_C$ equals the Appell hypergeometric function $F_4$.
 It is worth commenting that, if $\alpha\in(\mathbb{Z}_{\ge 0})^2$,
 then the above generalized Jacobi polynomial $P^{(\alpha,\beta)}_{2,k}(x)$
 can be expressed by the generalized hypergeometric function ${}_4F_3$
 (more precisely, see \eqref{for:P2} in Example~\ref{ex:d=2}).
 However, for general $d\ge 3$, we can not confirm such degeneracies.

 We further need a modification of special value of $P^{(\alpha,\beta)}_{d,k}(x)$.
 For $0\le h\le n$ and $z=(z_1,\ldots,z_d)\in(\mathbb{Z}_{\ge 0})^d$ with $|z|=h$, put 
\begin{align*}
 X^{(d)}_{M,h}(n;z)
=2(d-1)\delta_{h,0}+\frac{2n\bigl(-(2d-1)\bigr)^{\frac{n-h}{2}}}{n+h}\binom{h}{z_1,\ldots,z_d}
P^{(z,-1)}_{d,\frac{n-h}{2}}\Bigl(\frac{2d-3}{2d-1}\Bigr),
\end{align*}
 where $\delta_{h,0}$ is the Kronecker delta 
 and $\binom{h}{z_1,\ldots,z_d}=\frac{h!}{z_1!\cdots z_d!}$ denotes the multinomial coefficient.

 Write $N^{(d)}_M(n)=N_{\DT^{(d)}_M}(n)$ for short.
 The following is our main result.

\begin{thm}
\label{thm:main}
 Let $M=M^{(d)}=(m_1,\ldots,m_d)\in(\mathbb{Z}_{\ge 3})^d$.
 For $n\ge 3$, it holds that  
\begin{equation}
\label{for:PGTforDT}
 N^{(d)}_{M}(n)
=\Vert M\Vert\sum_{0\le h\le n \atop h\equiv n \!\!\!\!\! \pmod{2}}\sum_{z\in R^{(d)}_M(h)}m^{(d)}_{M}(z)X^{(d)}_{M,h}(n;z),
\end{equation}
 where, for $h\in\mathbb{Z}_{\ge 0}$,
 $R^{(d)}_M(h)$ is a finite set defined by 
\begin{align*}
 R^{(d)}_{M}(h)
=\left\{z=(z_1,\ldots,z_d)\in(\mathbb{Z}_{\ge 0})^d\,\left|
\begin{array}{l}
 z_1\ge \cdots \ge z_d, \ |z|=z_1+\cdots+z_d=h,\\
 \text{there exists $(y_1,\ldots,y_d)\in\mathbb{Z}^d$ such that, as a multiset,}\\
 \{z_1,\ldots,z_d\}=\{m_1|y_1|,\ldots,m_d|y_d|\}
\end{array}
\right.\right\}
\end{align*}
 and, for $z=(z_1,\ldots,z_d)\in R^{(d)}_{M}(h)$,
 $m^{(d)}_{M}(z)$ is a multiplicity of $X^{(d)}_{M,h}(n;z)$ given by  
\begin{align*}
 m^{(d)}_{M}(z)
=\#\left\{(y_1,\ldots,y_d)\in\mathbb{Z}^d\,\left|\,
 \text{as a multiset}, \ \{z_1,\ldots,z_d\}=\{m_1|y_1|,\ldots,m_d|y_d|\}
\right.\right\}.
\end{align*}
\end{thm}

 We remark that
 though we can not calculate the right hand side of \eqref{for:PGTforN} explicitly by hand in general,
 we can definitely calculate the one of \eqref{for:PGTforDT}
 because it is a finite sum of (special values of) polynomials $P^{(\alpha,\beta)}_{d,k}(x)$,
 whose coefficients are given concretely. 
 We also remark that above result is easily extended to general discrete tori
 corresponding to the groups $\mathbb{Z}^d/\Lambda\mathbb{Z}^d$
 where $\Lambda$ is an integer square matrix of size $d$ satisfying $\det \Lambda\ne 0$. 
 In this paper, for simplicity, we only consider such diagonal cases.

 We give a proof of Theorem~\ref{thm:main} in Section~\ref{sec:Spectral_zeta}.
 It is achieved by calculating the additive zeta function $\zeta^{(d)}_{M}(s)$ associated with $\DT^{(d)}_M$ in two ways:
 One is done by using the Ihara zeta function and
 the other is by the theory of the heat kernel obtained in \cite{ChintaJorgensonKarlsson2010}.
 In Section~\ref{sec:normalized_DT}, we investigate a special case,
 that is, the normalized discrete torus $\DT^{(d)}_{m}=\DT^{(d)}_{(m,\ldots,m)}$.
 We give some numerical examples and observations obtained from the examples.
 We see that, in this case, our formula \eqref{for:PGTforDT} seems to give a graph analogue of
 a refinement of the prime geodesic theorem for compact Riemannian manifolds of negative curvature,
 which counts closed geodesics lying in a fixed homology class. 

 Though there are plenty of papers about the error terms in the prime number theorem and
 prime geodesic theorems for Riemannian manifolds,
 there seems to be few results concerning the error terms in \eqref{for:PGTforPi}
 (see \cite{Nagoshi1999} for another kind of prime geodesic theorem with an error term for some special types of graphs).
 We finally remark that, because our formula \eqref{for:PGTforDT} is explicit,
 it may yield a description of the error terms of \eqref{for:PGTforPi} in the case of the discrete tori.

\section{Additive zeta functions}
\label{sec:Spectral_zeta}

 For a $(q+1)$-regular graph $X=(V,E)$,
 define
\[
 \zeta_X(s)
=\sum_{\lambda\in\Spec(\Delta_X)}\frac{1}{\lambda+s}.
\]
 We call $\zeta_X(s)$ an additive zeta function associated with $X$. 
 Here, $\Delta_X$ is the combinatoric Laplacian on $X$,
 that is, $\Delta_X=(q+1)I_{|V|}-A$ with $A$ being the (vertex) adjacency matrix of $X$.
 We sometimes understand that $\Delta_X$ is a linear operator 
 on the $\mathbb{C}$-vector space $L^2(X)=\{f:X\to\mathbb{C}\}$,
 endowed with the inner product $(f,g)=\sum_{x\in V}f(x)\overline{g(x)}$, $f,g\in L^2(X)$,
 acting by 
\[
 (\Delta_X f)(x)=(q+1)f(x)-\sum_{\{x,y\}\in E}f(y).
\]
 Here we denote $\{x,y\}\in E$ by the edge which connects vertices $x\in V$ and $y\in V$.
 The idea for obtaining the main result is to calculate the additive zeta function
 $\zeta^{(d)}_{M}(s)=\zeta_{\DT^{(d)}_M}(s)$ associated with $\DT^{(d)}_M$ in two ways;
 One is via the Ihara zeta function and the other is via the heat kernel on $\DT^{(d)}_M$.

 First, it is useful to give here the explicit descriptions 
 of the eigenvalues and the corresponding eigenfunctions of the combinatoric Laplacian $\Delta^{(d)}_M$ of $\DT^{(d)}_M$.
 Let 
\begin{align*}
 V^{(d)}_M
&=\prod^{d}_{j=1}\mathbb{Z}/m_j\mathbb{Z}
=\left\{(x_1,\ldots,x_d)\,\left|\,x_j\in\{0,1,\ldots,m_j-1\},\ j=1,2,\ldots,d\right.\right\}, \\
 (V^{(d)}_M)^{*}
&=\prod^{d}_{j=1}\frac{1}{m_j}\mathbb{Z}\Big/\mathbb{Z}
=\left\{(v_1,\ldots,v_d)\,\left|\,v_j\in\left\{0,\frac{1}{m_j},\ldots,\frac{m_j-1}{m_j}\right\},\ j=1,2,\ldots,d\right.\right\}.
\end{align*}
 Notice that $V^{(d)}_M$, $(V^{(d)}_M)^{*}$ are the sets of all vertices of $\DT^{(d)}_M$ and its dual $(\DT^{(d)}_M)^{*}$
 (see \cite{ChintaJorgensonKarlsson2015}), respectively.
 Moreover, for $x=(x_1,\ldots,x_d)\in V^{(d)}_M$ and $v=(v_1,\ldots,v_d)\in (V^{(d)}_M)^{*}$,
 put $(x,v)=x_1v_1+\cdots+x_dv_d$.
 
\begin{lem}
\label{lem:specData}
 For $v=(v_1,\ldots,v_d)\in(V^{(d)}_M)^{*}$, let 
\[
 \lambda_{v}=2d-2\sum^{d}_{k=1}\cos\bigl(2\pi v_k\bigr), \quad
 \phi_v(x)=\Vert M\Vert^{-\frac{1}{2}}e^{2\pi i(x,v)}.
\]
 Then, we have $\Spec(\Delta^{(d)}_M)=\bigl\{\lambda_v\,\bigl|\,v\in(V^{(d)}_M)^{*}\bigr\}$ and
 see that $\phi_v(x)$ is an orthonormal eigenfunction corresponding to $\lambda_{v}$, 
 that is, $\bigl\{\phi_v\bigr\}_{v\in (V^{(d)}_M)^{*}}$ forms an orthonormal basis of $L^{2}(\DT^{(d)}_M)$.
\end{lem}
\begin{proof}
 See \cite{ChintaJorgensonKarlsson2015}.
\end{proof}

\subsection{Ihara zeta functions}
\label{subsec:Ihara_zeta}

 Let $X=(V,E)$ be a $(q+1)$-regular graph.
 It is easy to see that 
\begin{equation}
\label{for:spec1}
 \zeta_X(s)
=\frac{d}{ds}\log\det\bigl(\Delta_X+sI_{|V|}\bigr).
\end{equation}
 We now show that, using the determinant expression
\begin{equation}
\label{for:detexpression_for_Ihara_A}
 Z_{X}(u)
=(1-u^2)^{|E|-|V|}\det\bigl(I_{|V|}-uA+qu^2 I_{|V|}\bigr)
\end{equation}
 of the Ihara zeta function $Z_{X}(u)$ associated with $X$,
 the additive zeta function $\zeta_X(s)$ can be written by the logarithmic derivative of $Z_{X}(u)$.
  
\begin{lem}
 Let $X=(V,E)$ be a $(q+1)$-regular graph.
 Then, it holds that 
\begin{equation}
\label{for:spectral_Ihara}
 \zeta_X(s)=|V|\frac{u_s}{1-u^2_s}+\frac{u_s}{1-qu^2_s}u_s\frac{d}{du}\log Z_{X}(u_s),
\end{equation} 
 where
\begin{equation}
\label{for:quadratic_trans}
 u_s=\frac{s+q+1\pm\sqrt{s^2+2(q+1)s+(q-1)^2}}{2q}.
\end{equation}
\end{lem}
\begin{proof}
 Substituting $A=(q+1)I_{|V|}-\Delta_X$ into \eqref{for:detexpression_for_Ihara_A}
 with the relation $2|E|=(q+1)|V|$, we have 
\begin{equation}
\label{for:detexpression_for_Ihara_Delta}
 \det\Bigl(\Delta_X+\frac{1-(q+1)u+qu^2}{u}I_{|V|}\Bigr)
=\left\{\Bigl(u(1-u^2)^{\frac{q-1}{2}}\Bigr)^{|V|}Z_{X}(u)\right\}^{-1}.
\end{equation}
 Now the desired formula \eqref{for:spectral_Ihara} is derived from \eqref{for:spec1}
 and \eqref{for:detexpression_for_Ihara_Delta} 
 with making a change of variable
\begin{equation}
\label{for:su1}
 s=\frac{1-(q+1)u_s+qu^2_s}{u_s},
\end{equation}
 that is, \eqref{for:quadratic_trans},
 together with the relation $\frac{d}{ds}u_s=-\frac{u^2_s}{1-qu^2_s}$.
\end{proof}

 From this lemma,
 because $\DT^{(d)}_{M}$ is a $2d$-regular graph (i.e., $q=2d-1$),
 one immediately obtains the following proposition.
  
\begin{prop}
 We have 
\begin{equation}
\label{for:DT_spectral_Ihara}
 \zeta^{(d)}_M(s)
=\Vert M\Vert\frac{u_s}{1-u^2_s}+\frac{u_s}{1-(2d-1)u^2_s}u_s\frac{d}{du}\log Z^{(d)}_{M}(u_s),
\end{equation} 
 where 
\[
 u_s=\frac{s+2d\pm\sqrt{s^2+4ds+4(d-1)^2}}{2(2d-1)}.
\]
\qed
\end{prop}
 
\begin{remark}
 We notice from \eqref{for:su1} that 
\begin{equation}
\label{for:su2}
 s+2d=\frac{1+(2d-1)u^2_s}{u_s}.
\end{equation}
 We will encounter this quadratic transformation in \S~\ref{subsec:HG}.
\end{remark} 
 
 Because $u_s\frac{d}{du}\log Z^{(d)}_{M}(u_s)=\sum^{\infty}_{n=1}N^{(d)}_M(n)u_s^n$,
 what we have to do next is to expand $\zeta^{(d)}_M(s)$ in a series in the variable $u_s$.
 
\subsection{Heat kernels}
\label{subsec:Heat_kernel}

 We next start from the fact that
 the additive zeta function $\zeta_X(s)$ can be expressed as 
 the Laplace transform of the theta function $\theta_X(t)$ of $X$ defined by 
\[
 \theta_X(t)
=\sum_{\lambda\in\Spec(\Delta_X)}e^{-\lambda t}.
\] 
 (This is why we call $\zeta_X(s)$ an additive zeta function.
 See \cite{JorgensonLangGoldfeld1994} and \cite{JorgensonLang2008} for this terminology).
 Actually, noticing that $\Spec(\Delta_X)\subset [0,2(q+1)]$, we have
\begin{equation}
\label{for:spec2}
 \zeta_X(s)
=\int^{\infty}_{0}e^{-st}\theta_X(t)dt \quad (\Re(s)>0).
\end{equation}
 From a general theory,
 we know that $\theta_X(t)$ is essentially given by the heat kernel $K_{X}(x,t):V\times \mathbb{R}_{>0}\to \mathbb{C}$
 on the graph $X$, which is the unique solution of the heat equation 
\[
\left\{
\begin{array}{l}
 \left(\Delta_X+\frac{\partial}{\partial t}\right)f(x,t)=0, \\[5pt]
 \displaystyle{\lim_{t\to 0}}f(x,t)=\delta_{o}(x).
\end{array}
\right.
\]
 Here, $o\in V$ is a fixed base point of $X$ and $\delta_{o}(x)$ is the Kronecker delta, that is,
 $\delta_{o}(x)=1$ if $x=o$ and $0$ otherwise.
 The following is well known.

\begin{lem}
 Let $X=(V,E)$ be a $(q+1)$-regular graph with and $o\in V$.
 Let $\phi_{\lambda}(x)$ be an  orthonormal eigenfunction of $\Delta_X$ with respect to $\lambda\in \Spec(\Delta_X)$. 
 Then, we have 
\begin{equation}
\label{for:HK_general}
 K_X(x,t)
=\sum_{\lambda\in\Spec(\Delta)}e^{-\lambda t}\overline{\phi_{\lambda}(o)}\phi_{\lambda}(x).
\end{equation}
\end{lem}
\begin{proof}
 See, e.g., \cite{Chung1997}.
\end{proof}

 Now, let us calculate $K^{(d)}_M(x,t)=K_{\DT^{(d)}_M}(x,t)$ and $\theta^{(d)}_{M}(t)=\theta_{\DT^{(d)}_M}(t)$.
 We notice that, in the case of studying $\DT^{(d)}_{M}$, 
 we always take a base point $o=(0,\ldots,0)\in V^{(d)}_{M}$. 
 From Lemma~\ref{lem:specData} and \eqref{for:HK_general}, we have 
\[
 K^{(d)}_M(x,t)
=\frac{1}{\Vert M\Vert}\sum_{v\in (\DT^{(d)}_M)^{*}}e^{-\lambda_v t}e^{2\pi i(x,v)}
\]
 and hence
\begin{equation}
\label{for:theta}
 \theta^{(d)}_M(t)
=\sum_{v\in (\DT^{(d)}_M)^{*}}e^{-\lambda_v t}=\Vert M\Vert K^{(d)}_M(o,t).
\end{equation}

 It is shown, for example in \cite{KarlssonNeuhauser2006} (see also \cite{Karlsson2012}), that 
 the heat kernel $K_{\mathbb{Z}}(x,t)$ on $\mathbb{Z}$ with $o=0$ is given by 
\[
 K_{\mathbb{Z}}(x,t)=e^{-2t}I_x(2t) \quad (x\in\mathbb{Z}),
\]
 where $I_x(t)$ is the $I$-Bessel function (or the modified Bessel function of the first kind) having the expansion  
\begin{equation}
\label{def:I_Bessel}
 I_x(t)
=\sum^{\infty}_{n=0}\frac{1}{n!\Gamma(n+x+1)}\Bigl(\frac{t}{2}\Bigr)^{2n+x}.
\end{equation}
 Hence, by the uniqueness of the heat kernel,
 we see that the heat kernel $K_{\mathbb{Z}^d}(x,t)$ on $\mathbb{Z}^d$ can be written as 
 $K_{\mathbb{Z}^d}(x,t)=\prod^{d}_{j=1}K_{\mathbb{Z}}(x_j,t)$ for $x=(x_1,\ldots,x_d)\in\mathbb{Z}^d$.  
 Moreover, periodizing this as a function on $V^{(d)}_{M}$, we have  
\begin{align*}
 K^{(d)}_M(x,t)
=\sum_{z\in\prod^{d}_{j=1}m_j\mathbb{Z}}K_{\mathbb{Z}^d}(x+z,t)
=\sum_{(z_1,\ldots,z_d)\in\prod^{d}_{j=1}m_j\mathbb{Z}}\prod^{d}_{j=1}K_{\mathbb{Z}}(x_j+z_j,t).
\end{align*}
 This shows from \eqref{for:theta} that 
\begin{align*}
 \theta^{(d)}_M(t)
=\Vert M\Vert K^{(d)}_M(o,t)
&=\Vert M\Vert\sum_{(z_1,\ldots,z_d)\in\prod^{d}_{j=1}m_j\mathbb{Z}}e^{-2dt}\prod^{d}_{j=1}I_{z_j}(2t)\\
&=\Vert M\Vert\sum^{\infty}_{h=0}\sum_{z=(z_1,\ldots,z_d)\in R^{(d)}_{M}(h)}m^{(d)}_{M}(z)e^{-2dt}\prod^{d}_{j=1}I_{z_j}(2t).
\end{align*}
 Here, one obtains the last equation above as follows:
 Let $W=\mathfrak{S}_d \ltimes (\mathbb{Z}/2\mathbb{Z})^d$ with $\mathfrak{S}_d$ being the symmetric group of degree $d$.
 We see that $W$ naturally acts on $\mathbb{Z}^d$
 and can take the set of all partitions of length less than or equal to $d$ 
 as a set of all representatives of the $W$-orbits on $\mathbb{Z}^d$.  
 For a partition $z$ of the length $l(z)\le d$,
 we have $m_M^{(d)}(z) = \# (\prod^d_{j=1} m_j \mathbb{Z}) \cap Wz$ where $Wz$ is the $W$-orbit of $z$.
 Then, for $h\in\mathbb{Z}_{\ge 0}$,
 we have $R_M^{(d)}(h)=\{z\vdash h\,|\,\text{$l(z)\le d$,\ $m_M^{(d)}(z) >0$}\}$ 
 and, because $\prod^d_{j=1} I_{z_j}(2t)$ is $W$-invariant (notice that $I_x(t)=I_{-x}(t)$),
 obtain the desired equation.

 Now, from \eqref{for:spec2}, it holds that  
\begin{align}
 \zeta^{(d)}_{M}(s)
&=\Vert M\Vert\sum^{\infty}_{h=0}\sum_{z\in R^{(d)}_{M}(h)}m^{(d)}_{M}(z)S^{(d)}_{z}(s+2d)
\nonumber\\
\label{for:spectral_HK}
&=\Vert M\Vert\sum^{\infty}_{h=0}\sum_{z\in R^{(d)}_{M}(h)}m^{(d)}_{M}(z)T^{(d)}_{z}(u_s),
\end{align}
 where, for $z=(z_1,\ldots,z_d)\in R^{(d)}_{M}(h)$, we put 
\begin{align*}
 S^{(d)}_z(x)
&=\int^{\infty}_{0}e^{-xt}\prod^{d}_{j=1}I_{z_j}(2t)dt \quad (\Re(x)>2d),\\
 T^{(d)}_z(u)
&=S^{(d)}_z\Bigl(\frac{1+(2d-1)u^2}{u}\Bigr).
\end{align*}
 Here, we have used the relation \eqref{for:su2}.
 We will study some properties of these functions in \S~\ref{subsec:HG}.

\subsection{Proof of Theorem~\ref{thm:main}}

 Let us give a proof of the main result.

\begin{proof}
[Proof of Theorem~\ref{thm:main}]
 From \eqref{for:DT_spectral_Ihara}, we have
\begin{equation}
\label{for:logderiIhara}
 u_s\frac{d}{du}\log Z^{(d)}_{M}(u_s)
=-\Vert M\Vert\frac{1-(2d-1)u^2_s}{1-u^2_s}+\frac{1-(2d-1)u^2_s}{u_s}\zeta^{(d)}_{M}(s).
\end{equation}
 Therefore, based on the equation \eqref{for:generating_funcition_of_N},
 one can obtain a formula for $N^{(d)}_M(n)$ by expanding the right hand side of \eqref{for:logderiIhara} in a series
 in the variable $u_s$.
 The first term on the right hand side of \eqref{for:logderiIhara} can be easily expanded as follows:
\begin{equation}
\label{for:FirstTerm}
 -\Vert M\Vert+2\Vert M\Vert(d-1)u^2_s+2\Vert M\Vert(d-1)\sum_{n\ge 4 \atop n\,:\,\text{even}}u^n_s.
\end{equation} 
 Moreover, 
 since we have from \eqref{for:spectral_HK} together with \eqref{for:G},
 which we will prove in \S~\ref{subsec:HG},
\begin{align*}
 \zeta^{(d)}_{M}(s)
&=
\Vert M\Vert\sum^{\infty}_{h=0}\sum^{\infty}_{k=0}
\bigl(-(2d-1)\bigr)^k
 \left\{\sum_{z\in R^{(d)}_{M}(h)}m^{(d)}_{M}(z)C_z 
 P^{(z,0)}_{d,k}\Bigl(\frac{2d-3}{2d-1}\Bigr)\right\}u^{h+2k+1}_s
\end{align*}
 with $C_z=\binom{|z|}{z_1,\ldots,z_d}$,
 the second term of \eqref{for:logderiIhara} can be written as 
\begin{align*}
 &
\Vert M\Vert\sum^{\infty}_{h=0}\sum^{\infty}_{k=0}
\bigl(-(2d-1)\bigr)^k
 \left\{\sum_{z\in R^{(d)}_{M}(h)}m^{(d)}_{M}(z)C_z 
 P^{(z,0)}_{d,k}\Bigl(\frac{2d-3}{2d-1}\Bigr)\right\}u^{h+2k}_s \\
&\ \ \ +\Vert M\Vert\sum^{\infty}_{h=0}\sum^{\infty}_{k=1}
\bigl(-(2d-1)\bigr)^k
 \left\{\sum_{z\in R^{(d)}_{M}(h)}m^{(d)}_{M}(z)C_z 
 P^{(z,0)}_{d,k-1}\Bigl(\frac{2d-3}{2d-1}\Bigr)\right\}u^{h+2k}_s \\
=&\Vert M\Vert\sum^{\infty}_{h=0}
 \left\{\sum_{z\in R^{(d)}_{M}(h)}m^{(d)}_{M}(z)C_z \right\}u^{h}_s \\
&\ \ \ +\Vert M\Vert\sum^{\infty}_{h=0}\sum^{\infty}_{k=1}
\bigl(-(2d-1)\bigr)^k\frac{h+2k}{h+k}
 \left\{\sum_{z\in R^{(d)}_{M}(h)}m^{(d)}_{M}(z)C_z 
 P^{(z,-1)}_{d,k}\Bigl(\frac{2d-3}{2d-1}\Bigr)\right\}u^{h+2k}_s \\
=&\Vert M\Vert\sum^{\infty}_{n=0}
 \left\{\sum_{z\in R^{(d)}_{M}(n)}m^{(d)}_{M}(z)C_z \right\}u^{n}_s \\
&\ \ \ +\Vert M\Vert\sum^{\infty}_{n=1}
\left\{\sum_{0\le h\le n-2 \atop h\equiv n \!\!\!\!\! \pmod{2}}\frac{2n\bigl(-(2d-1)\bigr)^{\frac{n-h}{2}}}{n+h}
 \sum_{z\in R^{(d)}_{M}(h)}m^{(d)}_{M}(z)C_z 
 P^{(z,-1)}_{d,\frac{n-h}{2}}\Bigl(\frac{2d-3}{2d-1}\Bigr)\right\}u^{n}_s.
\end{align*}
 Here, in the first equality,
 we have used the equation $P^{(z,0)}_{d,0}\bigl(\frac{2d-3}{2d-1}\bigr)=1$ and, for $k\ge 1$,
\[
 P^{(z,0)}_{d,k}(x)+P^{(z,0)}_{d,k-1}(x)
=\frac{|z|+2k}{|z|+k}P^{(z,-1)}_{d,k}(x),
\]
 which are easily seen from the definition.
 Let us write the coefficient of $u^n_s$ of the rightmost hand side of the above formula as, say, $C(n)$.
 Noticing that $R^{(d)}_{M}(0)=\{(0,\ldots,0)\}$ and $R^{(d)}_{M}(1)=R^{(d)}_{M}(2)=\emptyset$
 because $m_j\ge 3$ for all $j=1,2,\ldots,d$,
 we have 
\[
 C(0)
=\Vert M\Vert,\quad 
 C(1)
=0,\quad 
 C(2)
=\Vert M\Vert\cdot \bigl(-2(2d-1)\bigr)P^{(0,-1)}_{d,1}\Bigl(\frac{2d-3}{2d-1}\Bigr)
=-2\Vert M\Vert(d-1).
\]
 This shows that
 the second term on the right hand side of \eqref{for:logderiIhara} is expanded as follows: 
\begin{align}
\label{for:SecondTerm}
&\ \ \ \Vert M\Vert-2\Vert M\Vert(d-1)u^2_s\\
\nonumber
&+\Vert M\Vert\sum^{\infty}_{n=3}
\left\{\sum_{0\le h\le n \atop h\equiv n \!\!\!\!\! \pmod{2}}\frac{2n\bigl(-(2d-1)\bigr)^{\frac{n-h}{2}}}{n+h}
 \sum_{z\in R^{(d)}_{M}(h)}m^{(d)}_{M}(z)C_z 
 P^{(z,-1)}_{d,\frac{n-h}{2}}\Bigl(\frac{2d-3}{2d-1}\Bigr)\right\}u^{n}_s
\end{align}
 Combining \eqref{for:FirstTerm} and \eqref{for:SecondTerm},
 and noticing that $m^{(d)}_{M}(0,\ldots,0)=1$,
 we obtain the desired formula \eqref{for:PGTforDT}.
 This ends the proof.
\end{proof}

\subsection{A quadratic transformation $F^{(d)}_C$}
\label{subsec:HG}

 We here prove the following proposition,
 which contains a key formula for our results.

\begin{prop}
 Let $z=(z_1,\ldots,z_d)\in (\mathbb{Z}_{\ge 0})^d$.

\noindent
 $(1)$
 We have 
\begin{align}
\label{for:F_d1}
 S^{(d)}_{z}(x)
&=\frac{C_z}{x^{|z|+1}}
F^{(d)}_C\left(
\begin{array}{c}
 \frac{|z|}{2}+\frac{1}{2},\frac{|z|}{2}+1\\[3pt] 
 z_1+1,\ldots,z_d+1
\end{array}
;\,\frac{4}{x^2},\ldots,\frac{4}{x^2}
\right)\\
\label{for:F_d2}
&=\frac{1}{x^{|z|+1}}\sum^{\infty}_{n=0}A^{(d)}_z(n)\frac{(\frac{|z|}{2}+\frac{1}{2})_n(\frac{|z|}{2}+1)_n}{(h+1)_n}
\frac{(\frac{4}{x^2})^n}{n!},
\end{align}
 where $C_z=\binom{|z|}{z_1,\ldots,z_d}$ and
\[
 A^{(d)}_{z}(n)=\sum_{n_1,\ldots,n_d\ge 0 \atop n_1+\cdots +n_d=n}\binom{n}{n_1,\ldots,n_d}\binom{n+|z|}{n_1+z_1,\ldots,n_d+z_d}.
\]

\noindent
$(2)$
 We have 
\begin{equation}
\label{for:G}
 T^{(d)}_{z}(u)
=C_z u^{|z|+1}
 \sum^{\infty}_{k=0}P^{(z,0)}_{d,k}\Bigl(\frac{2d-3}{2d-1}\Bigr)\bigl(-(2d-1)u^2\bigr)^k,
\end{equation}
 where $P^{(\alpha,\beta)}_{d,k}(x)$ is the generalization of the Jacobi polynomial defined by \eqref{def:geneJacobi}.
 Moreover, if $\alpha\in(\mathbb{Z}_{\ge 0})^d$, then it can be written as   
\begin{equation}
\label{for:geneJacobi1}
 P^{(\alpha,\beta)}_{d,k}(x)
=\frac{1}{C_{\alpha}}\frac{(|\alpha|+1)_k}{k!}
\sum^{k}_{n=0}A^{(d)}_{\alpha}(n)\frac{(-k)_n(k+|\alpha|+\beta+1)_n}{(|\alpha|+1)_n}\frac{(\frac{1-x}{2})^n}{n!}.
\end{equation}
\end{prop}
\begin{proof}
 Put $h=|z|$.
 Using \eqref{def:I_Bessel}, we have 
\begin{align*}
 S^{(d)}_{z}(x)
&=\sum_{n=(n_1,\ldots,n_d)\in(\mathbb{Z}_{\ge 0})^d}\frac{1}{n_1!\cdots n_d!(n_1+z_1)!\cdots (n_d+z_d)!}
\int^{\infty}_{0}e^{-xt}t^{h+2|n|+1}\frac{dt}{t}\\
&=\frac{1}{x^{h+1}}\sum_{n=(n_1,\ldots,n_d)\in(\mathbb{Z}_{\ge 0})^d}
\frac{(h+2|n|)!}{n_1!\cdots n_d!(n_1+z_1)!\cdots (n_d+z_d)!}\frac{1}{x^{2|n|}}\\
&=\frac{1}{x^{h+1}}\frac{h!}{z_1!\cdots z_d!}\sum_{n=(n_1,\ldots,n_d)\in(\mathbb{Z}_{\ge 0})^d}
\frac{\bigl(\frac{h}{2}+\frac{1}{2}\bigr)_{|n|}\bigl(\frac{h}{2}+1\bigr)_{|n|}}{(z_1+1)_{n_1}\cdots (z_d+1)_{n_d}}
\frac{(\frac{4}{x^2})^{|n|}}{n_1!\cdots n_d!}.
\end{align*}
 Here, we have employed the identities 
 $(a+n)!=a!(a+1)_n$ and $(a+2n)!=a!2^{2n}\bigl(\frac{a}{2}+\frac{1}{2}\bigr)_n\bigl(\frac{a}{2}+1\bigr)_n$
 for $a,n\in\mathbb{Z}_{\ge 0}$. 
 Hence we obtain \eqref{for:F_d1}.
 The formula \eqref{for:F_d2} is easily obtained from \eqref{for:F_d1}. 

 We next concentrate on $G^{(d)}_{z}(u)$.
 From \eqref{for:F_d1}, it holds that 
\begin{align*}
 T^{(d)}_z(u)
=C_zu^{h+1}\sum_{n=(n_1,\ldots,n_d)\in(\mathbb{Z}_{\ge 0})^d}\frac{\bigl(\frac{h}{2}+\frac{1}{2}\bigr)_{|n|}
\bigl(\frac{h}{2}+1\bigr)_{|n|}}{(z_1+1)_{n_1}\cdots (z_d+1)_{n_d}}
\frac{2^{2|n|}u^{2|n|}}{n_1!\cdots n_d!}\bigl(1+(2d-1)u^2\bigr)^{-(h+1+2|n|)}.
\end{align*}
 The generalized binomial theorem yields 
\begin{align*}
 \bigl(1+(2d-1)u^2\bigr)^{-(h+1+2|n|)}
&=\sum^{\infty}_{l=0}\binom{h+l+2|n|}{l}\bigl(-(2d-1)u^2\bigr)^{l}\\
&=\sum^{\infty}_{l=0}\frac{(h+1)_{l+|n|}(l+|n|+h+1)_{|n|}}{l!2^{2|n|}\bigl(\frac{h}{2}+\frac{1}{2}\bigr)_{|n|}
\bigl(\frac{h}{2}+1\bigr)_{|n|}}\bigl(-(2d-1)u^2\bigr)^{l}
\end{align*}
 and hence 
\begin{align*}
& \ \ \ T^{(d)}_z(u)\\
&=C_zu^{h+1}\sum^{\infty}_{l=0}\left\{\sum_{n=(n_1,\ldots,n_d)\in(\mathbb{Z}_{\ge 0})^d}
\frac{(h+1)_{l+|n|}(l+|n|+h+1)_{|n|}}{(z_1+1)_{n_1}\cdots (z_d+1)_{n_d}}
\frac{u^{2(l+|n|)}}{l!n_1!\cdots n_d!}\right\}
\bigl(-(2d-1)\bigr)^{l}\\
&=C_zu^{h+1}\sum^{\infty}_{k=0}\left\{\sum_{n=(n_1,\ldots,n_d)\in(\mathbb{Z}_{\ge 0})^d \atop |n|\le k}
\frac{(h+1)_{k}(k+h+1)_{|n|}}{(z_1+1)_{n_1}\cdots (z_d+1)_{n_d}}
\frac{u^{2k}}{(|n|-k)!n_1!\cdots n_d!}\right\}
\bigl(-(2d-1)\bigr)^{k-|n|}\\
&=C_zu^{h+1}\sum^{\infty}_{k=0}\frac{(h+1)_{k}}{k!}\left\{\sum_{n=(n_1,\ldots,n_d)\in(\mathbb{Z}_{\ge 0})^d \atop |n|\le k}
\frac{(-k)_{|n|}(k+h+1)_{|n|}}{(z_1+1)_{n_1}\cdots (z_d+1)_{n_d}}
\frac{\bigl(\frac{1}{2d-1}\bigr)^{|n|}}{n_1!\cdots n_d!}\right\}
\bigl(-(2d-1)u^2\bigr)^{k}.
\end{align*}
 Therefore we obtain \eqref{for:G}.
 The equation \eqref{for:geneJacobi1} follows in the same manner as \eqref{for:F_d2}.
\end{proof}

\begin{example}
\label{ex:d=1}
 When $d=1$, we have respectively from \eqref{for:F_d1} and \eqref{for:G} 
\begin{align}
 S^{(1)}_h(x)
&=\frac{1}{x^{h+1}}
{}_2F_{1}\left(
\begin{array}{c}
 \frac{h}{2}+\frac{1}{2},\frac{h}{2}+1\\[3pt] 
 h+1
\end{array}
;\,\frac{4}{x^2}
\right),
\nonumber\\
\label{for:G1}
 T^{(1)}_h(u)
&=u^{h+1}\sum^{\infty}_{n=0}P^{(h,0)}_{1,k}(-1)(-u^2)^{k}=\frac{u^{h+1}}{1+u^2}.
\end{align}
 Here, in the last equality in \eqref{for:G1}, we have used the well-known formula 
\begin{equation}
\label{for:Jacobi-1}
 P^{(\alpha,\beta)}_{1,k}(-1)=(-1)^k\binom{k+\beta}{k}.
\end{equation}
 We remark that \eqref{for:G1} is also obtained from the Pfaff transformation 
\[
 {}_2F_{1}\left(
\begin{array}{c}
 \frac{a}{2},\frac{a+1}{2}\\[3pt] 
 a-b+1
\end{array}
;\,\frac{4x}{(1+x)^2}
\right)
=
(1+x)^a
{}_2F_{1}\left(
\begin{array}{c}
 a,b\\[3pt] 
 a-b+1
\end{array}
;\,x
\right)
\]
 for the Gauss hypergeometric functions ${}_2F_1$.
 From this, we can say that \eqref{for:G} is a kind of generalization
 of the Pfaff transformation for $F^{(d)}_C$ with special parameters.
\end{example}

\begin{example}
\label{ex:d=2}
 We next consider the case $d=2$.
 Let $z=(z_1,z_2)\in (\mathbb{Z}_{\ge 0})^2$.
 It is easy to see that 
\begin{equation}
\label{for:A2}
 A^{(2)}_{z}(n)=\binom{2n+|z|}{n+z_1,n+z_2}.
\end{equation}
 Hence, we have respectively from \eqref{for:F_d1} and \eqref{for:G} 
\begin{align*}
 S^{(2)}_z(x)
&=\frac{C_z}{x^{|z|+1}}
{}_4F_{3}\left(
\begin{array}{c}
 \frac{|z|}{2}+\frac{1}{2},\frac{|z|}{2}+\frac{1}{2},\frac{|z|}{2}+1,\frac{|z|}{2}+1\\[3pt] 
 |z|+1,z_1+1,z_2+1
\end{array}
;\,\frac{16}{x^2}
\right),\\
 T^{(2)}_z(u)
&=C_zu^{|z|+1}\sum^{\infty}_{k=0} P^{(z,0)}_{2,k}\Bigl(\frac{1}{3}\Bigr)(-3u^2)^{k}.
\end{align*}
 Here, ${}_pF_{q}$ is the generalized hypergeometric function defined by
\[
 {}_pF_{q}\left(\begin{array}{c}a_1,\ldots,a_p\\b_1,\ldots,b_q\end{array};x\right)
=
\sum^{\infty}_{n=0}\frac{(a_1)_{n}\cdots (a_p)_{n}}{(b_1)_{n}\cdots (b_q)_{n}}\frac{x^{n}}{n!}.
\] 
 Moreover, from \eqref{for:geneJacobi1} and \eqref{for:A2},
 we have for $\alpha=(\alpha_1,\alpha_2)\in(\mathbb{Z}_{\ge 0})^2$
\begin{equation}
\label{for:P2}
 P^{(\alpha,\beta)}_{2,k}(x)
=\frac{(|\alpha|+1)_k}{k!}
{}_4F_{3}\left(
\begin{array}{c}
 -k,k+|\alpha|+\beta+1,\frac{|\alpha|}{2}+\frac{1}{2},\frac{|\alpha|}{2}+1\\[3pt] 
 |\alpha|+1,\alpha_1+1,\alpha_2+1
\end{array}
;\,2(1-x)
\right).
\end{equation}
\end{example}

\section{Normalized discrete tori}
\label{sec:normalized_DT}

\subsection{Prime geodesic theorem for $\DT^{(d)}_m$}

 In this section,
 we consider a special case, that is, 
 a normalized discrete torus $\DT^{(d)}_m=\DT^{(d)}_{(m,\ldots,m)}$ for $m\ge 3$.
 In this case, the result obtained in the previous section becomes more simple form as below.
 Here, we put $N^{(d)}_{m}(n)=N^{(d)}_{(m,\ldots,m)}(n)$.
 
\begin{thm}
\label{thm:main2}
 For $n\ge 3$, it holds that  
\begin{equation}
\label{for:PGTforNDT}
 N^{(d)}_{m}(n)
=m^d\sum_{0\le h\le \frac{n}{m} \atop mh\equiv n \!\!\!\!\! \pmod{2}}\sum_{\mu\vdash h \atop l(\mu)\le d}
m(\mu)X^{(d)}_{m,h}(n;\mu),
\end{equation}
 where, for a partition $\mu=(\mu_1,\ldots,\mu_l)=(1^{m_1(\mu)}\cdots h^{m_{h}(\mu)})\vdash h$ of length $l(\mu)=l\le d$
 with $m_j(\mu)$ being the multiplicity of $j$ in $\mu$,
\begin{align}
\label{def:multiplicity}
 m(\mu)
&=2^{l}\binom{d}{l}u(\mu) \ \ \text{with} \ \ u(\mu)=\binom{l}{m_1(\mu),\ldots,m_h(\mu)}, \\
\nonumber
 X^{(d)}_{m,h}(n;\mu)
&=2(d-1)\delta_{h,0}+\frac{2n\bigl(-(2d-1)\bigr)^{\frac{n-mh}{2}}}{n+mh}\binom{mh}{m\mu_1,\ldots,m\mu_l}
P^{(m\mu,-1)}_{d,\frac{n-mh}{2}}\Bigl(\frac{2d-3}{2d-1}\Bigr).
\end{align}
\qed
\end{thm}

\begin{example}
 Consider the case $d=1$.
 Let us check the trivial result
\[
 N^{(1)}_{m}(n)
=
\begin{cases}
 0 & m\nmid n,\\
 2m & m\,|\,n
\end{cases}
\]
 from our formula. 
 First, we have  $m(\mu)=1$ if $\mu=0$ (the empty partition) and $2$ otherwise.
 Moreover, from \eqref{for:Jacobi-1}, it holds that   
\begin{align*}
 X^{(1)}_{m,h}(n;\mu)
&=
\begin{cases}
1 & m\,|\,n \ \ \text{and} \ \ h=\frac{n}{m}, \\
0 & \text{otherwise}.
\end{cases}
\end{align*}
 Hence, from \eqref{for:PGTforNDT}, one can actually obtain the desired formula. 
\end{example}

\begin{example}
 We next consider the case $d=2$.
 It holds that  
\begin{align*}
 N^{(2)}_{m}(n) 
&=m^2\sum_{0\le h\le \frac{n}{m} \atop mh\equiv n \!\!\!\!\! \pmod{2}}\sum_{\mu\vdash h \atop l(\mu)\le 2}
m(\mu)X^{(2)}_{m,h}(n;\mu)\\
&=m^2\left(\delta^{\mathrm{e}}_n X^{(2)}_{m,0}(n;0)+4\sum_{1\le h\le \frac{n}{m} \atop mh\equiv n \!\!\!\!\! \pmod{2}}
\left\{X^{(2)}_{m,h}\bigl(n;(h)\bigr)+\sum_{\mu\vdash h \atop l(\mu)=2}u(\mu)X^{(2)}_{m,h}(n;\mu)\right\}\right).
\end{align*}
 Here, $\delta^{\mathrm{e}}_n=1$ if $n$ is even and $0$ otherwise. 
 Notice that, for $\mu=(\mu_1,\mu_2)\vdash h$ with $l(\mu)=2$, $u(\mu)=1$ if $\mu_1=\mu_2$ and $2$ otherwise.
 Using \eqref{for:geneJacobi1}, we have 
\[
 X^{(2)}_{m,0}(n;0)
=2+2(-3)^{\frac{n}{2}}
{}_3F_{2}\left(
\begin{array}{c}
 -\frac{n}{2},\frac{n}{2},\frac{1}{2}\\[3pt] 
 1,1
\end{array}
;\,\frac{4}{3}
\right).
\]
 Moreover, for $h\ge 1$ and $\mu=(\mu_1,\mu_2)\vdash h$ with $l(\mu)=2$, letting $k=\frac{n-mh}{2}$, we have
\begin{align*}
 X^{(2)}_{m,h}(n;(h))
&=\frac{4n(-3)^{k}}{n+mh}\frac{(mh+1)_k}{k!}
{}_4F_{3}\left(
\begin{array}{c}
 -k,k+mh,\frac{mh}{2}+\frac{1}{2},\frac{mh}{2}+1\\[3pt] 
 mh+1,mh+1,1
\end{array}
;\,\frac{4}{3}
\right),\\
 X^{(2)}_{m,h}\bigl(n;(\mu_1,\mu_2)\bigr)
&=\frac{4n(-3)^{k}}{n+mh}\binom{mh}{m\mu_1,m\mu_2}\frac{(mh+1)_k}{k!}
{}_4F_{3}\left(
\begin{array}{c}
 -k,k+mh,\frac{mh}{2}+\frac{1}{2},\frac{mh}{2}+1\\[3pt] 
 m\mu_1+1,m\mu_2+1,mh+1
\end{array}
;\,\frac{4}{3}
\right).
\end{align*}
 For example,
 let us consider the case $m=3$ and $n=6$.
 Since  
\[
 N^{(2)}_{3}(6) 
=9\left(
X^{(2)}_{3,0}(6;0)
+4X^{(2)}_{3,2}\bigl(6;(2)\bigr)
+4X^{(2)}_{3,2}\bigl(6;(1,1)\bigr)
\right)
\]
 with 
\begin{align*}
 X^{(2)}_{3,0}(6;0)
&=2+2(-27)
{}_3F_{2}\left(
\begin{array}{c}
 -3,3,\frac{1}{2}\\[3pt] 
 1,1
\end{array}
;\,\frac{4}{3}
\right)
=2+2(-27)\Bigl(-\frac{11}{27}\Bigr)
=24, \\
 X^{(2)}_{3,2}\bigl(6;(2)\bigr)
&={}_4F_{3}\left(
\begin{array}{c}
 0,6,\frac{7}{2},4\\[3pt] 
 7,7,1
\end{array}
;\,\frac{4}{3}
\right)
=1,\\
 X^{(2)}_{3,2}\bigl(6;(1,1)\bigr)
&=\binom{6}{3}
{}_4F_{3}\left(
\begin{array}{c}
 0,6,\frac{7}{2},4\\[3pt] 
 4,4,7
\end{array}
;\,\frac{4}{3}
\right)
=20,
\end{align*}
 we have 
\[
 N^{(2)}_{3}(6) 
=9\bigl(1\cdot 24+4\cdot 1+4\cdot 20\bigr)
=9\cdot 108
=972.
\]
 For the other values of $X^{(2)}_{3,h}(n;\mu)$ with $n\le 10$, see the table below.

\begin{table}[!h]
\noindent
{\small
\begin{center}
\renewcommand{\arraystretch}{1.3}
\begin{tabular}{c||c|c|c|c||c}
 $n$  & $h=0$ & $h=1$ & $h=2$ & $h=3$ & $\frac{N^{(2)}_3(n)}{3^2}$  \\
\hline 
\hline 
 $3$ & & $X(1)=1$ & & & 4 \\
\hline
 $4$ & $X(0)=8$ & & & & 8 \\
\hline
 $5$ & & $X(1)=10$ & & & 40 \\
\hline
 $6$ & $X(0)=24$ & & $X(2)=1$, $X(1^2)=20$ & & 108 \\
\hline
 $7$ & & $X(1)=42$ & & & 168 \\
\hline
 $8$ & $X(0)=216$ & & $X(2)=40$, $X(1^2)=80$ & & 696 \\
\hline
 $9$ & & $X(1)=414$ & & $X(3)=1$, $X(21)=84$  & 2332 \\
\hline
 $10$ & $X(0)=1520$ & & $X(2)=420$, $X(1^2)=840$ & & 6560 
\end{tabular}
\end{center}
}
\caption{The values of $X(\mu)=X^{(2)}_{3,h}(n;\mu)$}
\end{table}
\end{example}

\subsection{Some observations}

 From the table above, we first expect the following.
 
\begin{conj}
\label{conj:Xinteger}
 It holds that 
 $X^{(d)}_{m,h}(n;\mu)\in\mathbb{Z}_{\ge 0}$.
\end{conj}
 
 If Conjecture~\ref{conj:Xinteger} is true,
 then we can expect that the summand $X^{(d)}_{m,h}(n;\mu)$ in the righthand side of \eqref{for:PGTforNDT} itself
 counts something special type of cycles in $\DT^{(d)}_{m}$, as we will see below:
 Let us subdivide $N^{(d)}_{m}(n)$ into small pieces by the following manner. 
 Fix $o\in V^{(d)}_m$. 
 We notice that there is an one-to-one correspondence between a cycle $C$ in $\DT^{(d)}_{m}$
 starting from and ending to $o$ of length $n$ and a path $\overline{C}$ in $\mathbb{Z}^d$
 starting from the origin $(0,\ldots,0)$ and ending to $(mp_1,\ldots,mp_d)$
 for some $(p_1,\ldots,p_d)\in\mathbb{Z}^d$ of length $n$. 
 Let us call a path $\overline{C}$ in $\mathbb{Z}^d$ {\it reduced modulo $m$}
 if the corresponding cycle $C$ in $\DT^{(d)}_{m}$ is reduced.
 Let $\mathrm{RC}^{(d)}_m(n)$ be the set of all reduced cycles in $\DT^{(d)}_{m}$ starting from and ending to $o$ of length $n$
 and, for $p=(p_1,\ldots,p_d)\in\mathbb{Z}^d$, $\overline{\mathrm{RP}}^{(d)}_{m}(n;p)$
 the set of all reduced paths modulo $m$ in $\mathbb{Z}^d$
 starting from the origin $(0,\ldots,0)$ and ending to $(mp_1,\ldots,mp_d)$ of length $n$. 
 It is clear that 
\begin{align}
 N^{(d)}_{m}(n)
&=m^d\#\mathrm{RC}^{(d)}_m(n)\nonumber\\
&=m^d\sum_{p\in\mathbb{Z}^d}\#\overline{\mathrm{RP}}^{(d)}_{m}\bigl(n;p\bigr)\nonumber\\
\label{for:subdivision_of_N}
&=m^d\sum_{0\le h\le \frac{n}{m} \atop mh\equiv n \!\!\!\!\! \pmod{2}}\sum_{\mu\vdash h \atop l(\mu)\le d}
m(\mu)N^{(d)}_{m,h}(n;\mu),
\end{align}
 where, for $\mu=(\mu_1,\ldots,\mu_d)\vdash h$ of length $l(\mu)=l\le d$, $m(\mu)$ is defined in \eqref{def:multiplicity} and 
 $N^{(d)}_{m,h}(n;\mu)$ is the number of all reduced paths modulo $m$ in $\mathbb{Z}^d$
 starting from the origin $(0,\ldots,0)$ and ending to $(m\mu_1,\ldots,m\mu_d)$ of length $n$. 
 Notice that $m^d$ represents the number of choices of the starting points.
 Now, it is natural from \eqref{for:PGTforNDT} and \eqref{for:subdivision_of_N} to expect the following.
 
\begin{conj}
\label{conj:XN}
 It holds that 
 $X^{(d)}_{m,h}(n;\mu)=N^{(d)}_{m,h}(n;\mu)$.
\end{conj}

 It is clear that Conjecture~\ref{conj:Xinteger} follows from Conjecture~\ref{conj:XN}
 because $N^{(d)}_{m,h}(n;\mu)\in\mathbb{Z}_{\ge 0}$. 
 
\begin{example}
\label{ex:d2m3n6}
 The following figures~1, 2 and 3 support that Conjecture~\ref{conj:XN} is true for the case $d=2$, $m=3$ and $n=6$,
 that is,
 $N^{(2)}_{3,0}(6;(0))=X^{(2)}_{3,0}(6;(0))=24$,
 $N^{(2)}_{3,2}(6;(2))=X^{(2)}_{3,2}(6;(2))=1$ and
 $N^{(2)}_{3,2}(6;(1,1))=X^{(2)}_{3,2}(6;(1,1))=20$.
 Here, the black dots in the figures represent the lattices points.
 In particular, the big black dots denote points of the form of $(mp_1,mp_2)$ for some $(p_1,p_2)\in\mathbb{Z}^2$. 
 By the same manner, we have already checked that the equation $X^{(d)}_{m,h}(n;\mu)=N^{(d)}_{m,h}(n;\mu)$ holds for $n\le 10$.
\begin{figure}[!h]
\begin{center}
 \includegraphics[clip,width=140mm]{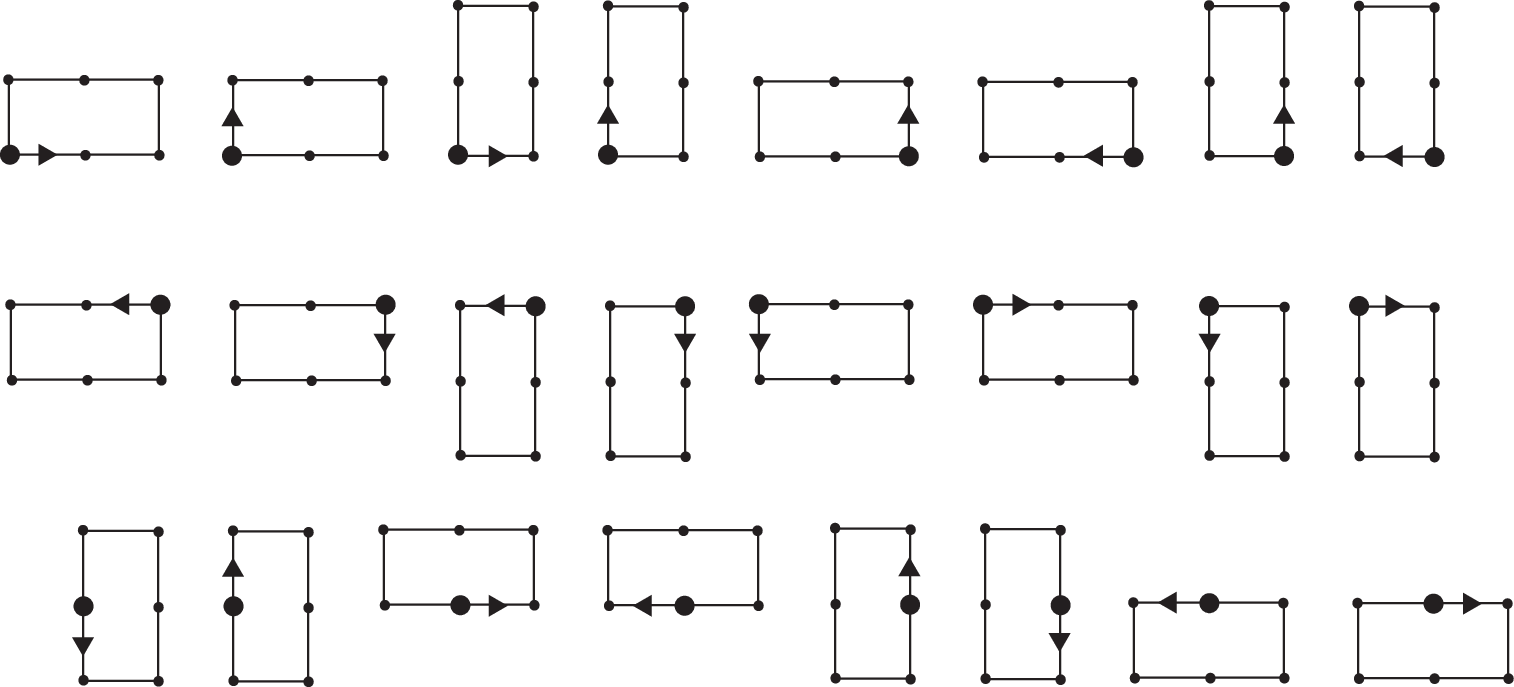}
\caption{$N^{(2)}_{3,0}(6;(0))=24$;
 it is the number of all reduced paths $\overline{C}$ modulo $3$ in $\mathbb{Z}^2$
 starting from $(0,0)$ and ending to $(0,0)$ of length $6$.}
\end{center}
\begin{center}
 \includegraphics[clip,width=50mm]{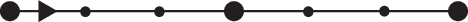}
\caption{$N^{(2)}_{3,2}(6;(2))=1$;
 it is the number of all reduced paths $\overline{C}$ modulo $3$ in $\mathbb{Z}^2$
 starting from $(0,0)$ and ending to $(6,0)$ of length $6$.}
\end{center}
\begin{center}
 \includegraphics[clip,width=150mm]{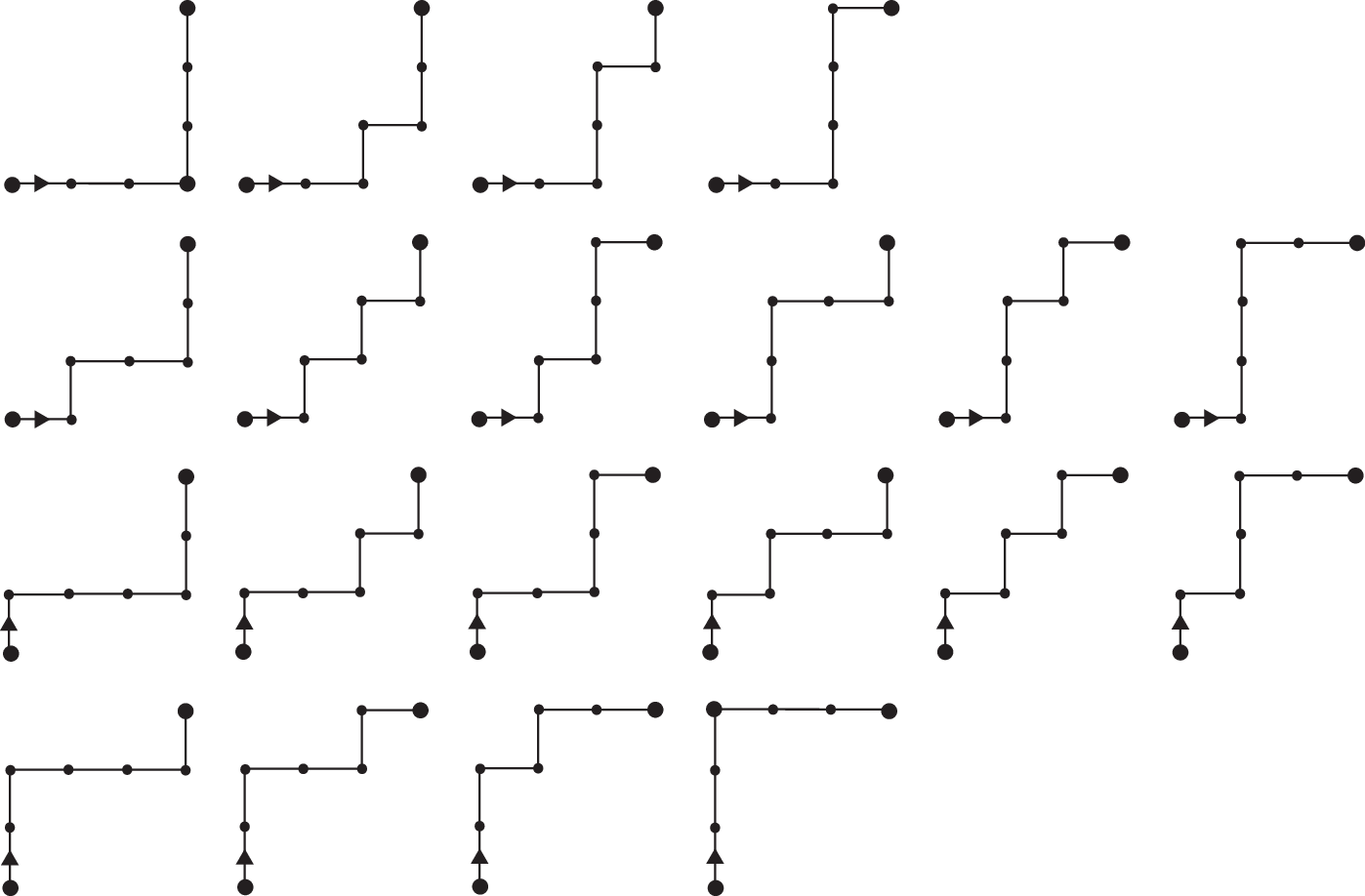}
\caption{$N^{(2)}_{3,2}(6;(1,1))=20$;
 it is the number of all reduced paths $\overline{C}$ modulo $3$ in $\mathbb{Z}^2$
 starting from $(0,0)$ and ending to $(3,3)$ of length $6$.}
\end{center}
\end{figure}
\end{example}

 Let $M$ be a compact Riemannian manifold of negative curvature.
 It is known that there exist countably infinitely many closed geodesics in $M$ and 
 that a closed geodesic in $M$ corresponds to
 a unique non-trivial conjugacy class $\mathrm{Conj}(\gamma)$ of $\gamma\in\pi_1(M)$
 where $\pi_1(M)$ is the fundamental group of $M$.
 Let us write the closed geodesic corresponding to $\mathrm{Conj}(\gamma)$ as $C_{\gamma}$.
 Let $N_M(x)$ be the number of all closed geodesics in $M$ of length $\le x$.
 Then, we have the following prime geodesic theorem for $M$ (\cite{Margulis1969});
\[
 N_M(x)\sim \frac{e^{hx}}{hx} \quad (x\to\infty),
\]
 where $h>0$ is the topological entropy of the geodesic flow over $M$.
 This is an analogue of the classical prime number theorem.
 Moreover, we have also an analogue of the Dirichlet theorem on arithmetic progressions,
 which counts the closed geodesics lying in a fixed homology class:
 Let $H_1(M,\mathbb{Z})$ be the first homology group over $\mathbb{Z}$ of $M$ and
 $\phi:\pi_1(M)\to H_1(M,\mathbb{Z})=\pi_1(M)^{\mathrm{ab}}=\pi_1(M)/[\pi_1(M),\pi_1(M)]$ be the natural projection. 
 Here, $[\pi_1(M),\pi_1(M)]$ is the commutant subgroup of $\pi_1(M)$.
 For a fixed $\alpha\in H_1(M,\mathbb{Z})$,
 let $N_M(x;\alpha)$ be the number of all closed geodesics $C=C_{\gamma}$ in $M$ of length $\le x$
 satisfying $\phi(\gamma)=\alpha$.
 Then, it is shown in \cite{AdachiSunada1987,PhillipsSarnak1987,Lalley1989} that
 there exists a constant $C>0$, not depending on $\alpha$, such that 
\begin{equation}
\label{for:PGThomology}
 N_M(x;\alpha)\sim C\frac{e^{hx}}{x^{\frac{b}{2}+1}} \quad (x\to\infty).
\end{equation} 
 Here, $b\in\mathbb{Z}_{\ge 0}$ is the first Betti number of $M$, that is, the rank of $H_1(M,\mathbb{Z})$.

 Now, one may regard the claim in Conjecture~\ref{conj:XN}
 as a graph analogue of \eqref{for:PGThomology}
 because the partition $\mu\vdash h$ of length at most $d$ appeared in $X^{(d)}_{m,h}(n;\mu)$
 can be regarded as an element of $\mathbb{Z}^d\simeq H_1(\RT^{(d)},\mathbb{Z})\,(\simeq \pi_1(\RT^{(d)}))$
 with $\RT^{(d)}$ being the real torus of dimension $d$.
 Let us see this in the situation of Example~\ref{ex:d2m3n6}.
 Write $H_1(\RT^{(d)},\mathbb{Z})=\mathbb{Z}\alpha+\mathbb{Z}\beta$
 with the standard basis $\alpha,\beta$ respectively corresponding to the meridian and longitude of the torus.
 Then, the conjecture says that
 $X^{(2)}_{3,0}(6;(0))$, $X^{(2)}_{3,2}(6;(2))$ and $X^{(2)}_{3,2}(6;(1,1))$
 count cycles in $\DT^{(2)}_3$ lying in homology classes $0$ (the null homology class),
 $2\alpha$ (or $2\beta$ by the symmetry) and $\alpha+\beta$, respectively. 

\section*{Acknowledgment} 

 The author would like to thank Professor Hiroyuki Ochiai
 for carefully reading the manuscript and giving many helpful comments. 
 Moreover, the author also thanks the referee for several useful and variable comments which improve the paper.
  

\noindent
\textsc{Yoshinori YAMASAKI} \\
 Graduate School of Science and Engineering, Ehime University, \\
 Bunkyo-cho, Matsuyama, 790-8577 JAPAN. \\
 \texttt{yamasaki@math.sci.ehime-u.ac.jp}

\end{document}